\documentclass[11pt, reqno]{amsart}
\usepackage{amssymb,amsmath,latexsym,amsthm}
\usepackage[all]{xy}
\usepackage{pstricks}\usepackage{stmaryrd}\usepackage{hyperref}

\newcommand{\comments}[1]{}
\newcommand\ff{\mathcal{F}}
\newcommand\gc{\mathcal{G}}
\newcommand\ee{\mathcal{E}}
\newcommand\cc{\mathbb{C}}
\newcommand\zz{\mathbb{Z}}
\newcommand\oo{\mathcal{O}}
\newcommand\Hom{{\rm Hom}}
\newcommand\Ext{{\rm Ext}}
\newcommand\pp{\mathbb{P}}
\newcommand\po{{\mathbb{P}^1}}
\newcommand\op{\mathcal{O}_{\mathbb{P}^1}}
\newcommand\hk{{\mathbb{F}_k}}

\newcommand\mo{$\pi_*\oo_Y$-module }
\newcommand\linecl{{[\mathbb{P}^1]}}

\newtheorem{thm}{Theorem}[section]
\newtheorem{conj}{Conjecture}[section]
\newtheorem{lem}[thm]{Lemma}
\newtheorem{prop}[thm]{Proposition}
\newtheorem{cor}[thm]{Corollary}
\theoremstyle{definition}
\newtheorem{ex}[thm]{Example}
\newtheorem{defn}[thm]{Definition}
\newtheorem{convention}[thm]{Convention}
\theoremstyle{remark}
\newtheorem{rmk}[thm]{Remark}

\begin{document}

\title{Genus zero BPS invariants for local $\mathbb{P}^1$}
\author{Jinwon Choi}
\address{Department of Mathematics, University of Illinois at
Urbana-Champaign} \email{choi29@illinois.edu}

\begin{abstract}
We study the equivariant version of the genus zero BPS invariants of
the total space of a rank 2 bundle on $\mathbb{P}^1$ whose
determinant is $\mathcal{O}_{\mathbb{P}^1}(-2)$. We define the equivariant genus
zero BPS invariants by the residue integrals on the moduli space of
stable sheaves of dimension one as proposed by Sheldon Katz
\cite{katz_gv}. We compute these invariants for low degrees by
counting the torus fixed stable sheaves. The results agree with the
prediction in local Gromov-Witten theory studied in \cite{localgw}.
\end{abstract}

\comments{
\begin{keyword}
BPS invariant \sep moduli space \sep equivariant sheaf \sep toric
variety
\MSC 14N35
\end{keyword}
}

\maketitle

\section{Introduction}
The 0-pointed genus $g$ Gromov-Witten invariant for a Calabi-Yau
three-fold $X$ in the curve class $\beta\in H_2(X,\zz)$ is defined as
the degree of the virtual cycle of the moduli space of stable maps
to $X$.
\[ N_{\beta}^g(X):= {\rm deg }[\overline{M}_{g,0}(X,\beta)]^{\rm
vir}.\] By the BPS state counts in M-theory, Gopakumar and Vafa
\cite{gv} proposed integer-valued invariants $n_\beta^g(X)$ of $X$,
called the \emph{BPS invariants}, which are related to the
Gromov-Witten invariants by the Gopakumar-Vafa formula
\[\sum_{\beta, g} N_{\beta}^g(X) q^\beta \lambda^{2g-2}
=\sum_{\beta, g, k}n^g_\beta(X) \frac{1}{k}\left(2
\sin\left(\frac{k\lambda}{2}\right)\right) ^{2g-2}q^{k\beta}.\]

A priori, the BPS invariants defined by the above formula are rational
numbers because the Gromov-Witten invariants are rational numbers.
The \emph{integrality conjecture} is an assertion that they are
integers.

The genus zero part of the above formula is
\begin{equation}\label{gv}
 N_\beta^0(X)=\sum_{m|\beta} \frac{n_{\beta/m}^0(X)}{m^3}.
\end{equation}

Katz \cite{katz_gv} proposed a mathematical definition for
the genus zero BPS invariants. He considered the Donaldson-Thomas
type invariants of the moduli space of stable sheaves of class $\beta$ and Euler characteristic one.
He showed \eqref{gv} holds for embedded contractible rational
curves. Shortly thereafter, Li and Wu \cite{li_gv}
studied K3 fibred local Calabi-Yau three-folds and verified
\eqref{gv} for curve classes $d\beta_0$ where $d\le 5$ and $\beta_0$
generates the Picard group of the central fiber.

Bryan and Gholampour \cite{brygho} studied the
equivariant version of BPS invariant for the resolution of ADE
polyhedral singularities $\cc^3/G$. As the moduli space of sheaves
is noncompact, the virtual cycle is not well defined. But using a
natural $\cc^*$-action induced from an action on $\cc^3/G$, they
defined the BPS invariants via equivariant residue integrals of the
virtual cycle at the fixed locus. They proved the equivariant BPS invariants so defined are in agreement with the prediction of equivariant Gromov-Witten theory via formula \eqref{gv}. In this paper, we follow this
approach and study the equivariant version of BPS invariants for
local $\po$.

Local $\po$ in this paper is the total space $X$ of rank 2 vector bundle $$E\simeq
\op(k)\oplus\op(-2-k)$$ on $\po$. Since ${\rm det}E\simeq
K_\po\simeq \op(-2)$, $X$ is a noncompact Calabi-Yau three-fold in
the sense that its canonical bundle is trivial.

The Gromov-Witten theory of $X$ is studied by Bryan and Pandharipande \cite{localgw}. They used the natural
$(\cc^*)^2$-action on $X$ via scalar multiplication on each fiber,
and computed residue Gromov-Witten invariants by localization and
degeneration methods. After taking anti-diagonal subtorus of
$(\cc^*)^2$, they got a closed formula for the Gromov-Witten
partition function \cite[Cor. 7.2]{localgw}.

We use a torus action for which the torus also acts nontrivially on
the base curve $\po$. By Calabi-Yau condition, our action restricts
to the action of their anti-diagonal subtorus(Section 2). So, we expect the genus 0 Gopakumar-Vafa formula \eqref{gv} holds for the
total space $X$ of $E$.
\begin{conj}[Equivariant GW/GV correspondence]\label{GV}
For $\beta=d[\po]\in H_2(X,\zz)$, let $N_{d}^{\rm GW}(k)$ be the
genus 0 local Gromov-Witten invariant computed in \cite{localgw} and
$n_d(k)$ be the equivariant local BPS invariant defined by the
residue integral in Definition~\ref{bps}. Then, the Gopakumar-Vafa
formula
$$  N_{d}^{\rm GW}(k)=\sum_{m|d} \frac{n_{d/m}(k)}{m^3}$$ holds.
\end{conj}

We prove this conjecture for $d=1,2,$ and $3$ for any $k$ and for
$d=4$ for $k\le 100$. We show that the moduli space of stable sheaves on
$X$ is smooth, and count the torus fixed sheaves using the
classification of equivariant sheaves in \cite{kool}.

This paper is organized as follows. In Section~\ref{bpsdef}, we define the equivariant genus zero BPS invariants.
In Section~\ref{equisheaf}, we review the classification of equivariant sheaves in \cite{kool}. We then classify stable equivariant sheaves on local $\po$ in Section~\ref{classification}. In Sections~\ref{bpscounting} and \ref{degree4}, we compute the BPS invariants for low degrees by counting equivariant sheaves. Finally, in Section~\ref{residue}, we give a proof that the moduli space of sheaves on $X$ is smooth and justify our computation.

\bigskip\noindent \emph{Acknowledgements.}
This work is part of my doctoral dissertation at University of
Illinois at Urbana-Champaign. I would like to thank my advisor
Sheldon Katz for invaluable discussions and many suggestions for
improvements. I would also like to thank Martijn Kool for kindly
explaining his work to me. Finally, I thank the referees for helpful and careful comments and suggestion. This work was partially supported by NSF
grants DMS-0244412 and DMS-0555678.

\section{Equivariant Local BPS Invariant}\label{bpsdef}
Let $k$ be an integer with $k\geq
-1$. Let $X={\rm Spec}({\rm Sym}(E^*))$ be the total space of a
rank 2 bundle
$$E=\op(k)\oplus\op(-2-k)$$ on $\po$. As a toric variety, $X$ contains a torus
$T'=(\cc^*)^3$ and has two $T'$-invariant affine open sets
isomorphic to $\cc^3$. The transition map is
$$(z_1,z_2,z_3)\mapsto (z_1^{-1},z_1^{-k}z_2,z_1^{2+k}z_3).$$
Here, the torus $T'$ acts by
\begin{equation}\label{eq:Tactionlocalp1}
(t_1,t_2,t_3).(z_1,z_2,z_3)=(t_1z_1,t_2z_2,t_3z_3).
\end{equation}
We will consider the action of the subtorus $$T=\{(t_1,t_2,t_3)\in
T' \colon t_1t_2t_3=1\}$$ which preserves the canonical Calabi-Yau
form \cite{mnop1}.

We define the dimension of a sheaf by the dimension of its support. A sheaf is called \emph{pure} of dimension $d$ if any nontrivial subsheaf has dimension $d$. Let $L$ be the pullback of $\oo_\po(1)$ to $X$. We construct the moduli space of $L$-stable pure dimension one sheaves $\ff$ such that the support of $\ff$ has class $d\linecl\in H_2(X)$ and $\chi(\ff)=1.$ Although $X$ is not projective, we may define Hilbert polynomial and Gieseker semistability for such sheaves.

For a sheaf $\ff$ whose support is in class $d\linecl\in H_2(X)$, we define the \emph{multiplicity} $r(\ff)$ by $r(\ff)=d$ and the Hilbert polynomial by
\begin{equation}\label{eq:hilblocalp1}
P_\ff(n)=r(\ff)n+\chi(\ff).
\end{equation}

A pure dimension one sheaf $\ff$ is called \emph{(Gieseker) semistable} with respect to $L$ if for any proper nonzero subsheaf $\gc$, we have
\[\frac{\chi(\gc)}{r(\gc)} \le \frac{\chi(\ff)}{r(\ff)}.\] Stable pure dimension one sheaf is defined with the strict inequality. For details and a construction of the moduli space of semistable
sheaves, we refer to \cite{hl}.

We consider the moduli space of $L$-(semi)stable coherent sheaves of
pure dimension 1 on $X$
$$ M_d(k)=\{ \ff \colon P_\ff=dn+1, \ff
\text{ is $L$-(semi)stable}\}.$$ 
By the condition $\chi(\ff)=1$, semistability agrees with stability.
So, there exists a perfect obstruction theory on $M_d(k)$
\cite{tho}. Unfortunately, since $M_d(k)$ is not compact, the virtual cycle for $M_d(k)$
is not well defined.

The $T$-action \eqref{eq:Tactionlocalp1} on $X$ induces a $T$-action
on the moduli space. Thus, we may define an equivariant version of invariant by means of the virtual localization.
\begin{defn}
  Let $X$ be a toric variety.
  \begin{enumerate}
\item   A sheaf $\ff$ on $X$ called \emph{$T$-fixed} if $t^*\ff\simeq \ff$.
\item Let $\sigma\colon T\times X\to X$ be $T$-action on $X$ and $p\colon T\times X\to X$ be the projection. A sheaf $\ff$ is \emph{$T$-equivariant} if we have an isomorphism $\Phi\colon \sigma^*\ff \to p^*\ff$ satisfying the cocycle condition
   \[ (\mu \times 1_X)^*\Phi= p_{23}^*\Phi\circ (1_T\times \sigma)^*\Phi,  \]
where $\mu\colon T\times T\to T$ is the multiplication map and $p_{23}\colon T\times T\times X \to T\times X$ is the projection to the second and the third factors.
\end{enumerate}
\end{defn}

We first note that the stable sheaves on $X$ are actually supported on a smaller subspace.

\begin{lem}\label{sup}
Denote by $Y$ the total space of $\oo_\po(k)$. If $\ff\in M_d(k)$,
then the scheme theoretic support of $\ff$ is in $Y$.
\end{lem}
\begin{proof}
The ideal sheaf of $Y$ is $L^{2+k}$. We have an exact sequence
$$
\xymatrix @1{%
\ff\otimes L^{2+k}  \ar[r] &\ff \ar[r] &\ff|_Y  \ar[r] & 0.}
$$
Since $2+k$ is a positive number, by the stability of $\ff$, the
first map is zero, and hence the map $\ff\to\ff|_Y$ is an
isomorphism.
\end{proof}
So, we can consider $\ff$ as a sheaf on $Y$. We will also denote by
$L$ the pullback of $\op(1)$ to $Y$. Then, $M_d(k)$ is the moduli
space of $L$-stable sheaves on $Y$.
Note that the zero section $\po$ in $Y$ is the only compact $T$-invariant curve in $Y$. Hence, if a sheaf $\ff$ is $T$-fixed, its reduced support must be $\po$. Then we have the following.

\begin{lem}
The fixed point locus of the induced $T$-action on $M_d(k)$ is compact.
\end{lem}
\begin{proof}
We will see in Section \ref{residue} that we can embed $M_d(k)$ into a compact moduli space via an embedding of $Y$ into the Hirzebruch surface $\hk$. The torus fixed locus supported on $\po$ is the same, and hence it is compact.
\end{proof}

Therefore, we can define an invariant by the residue integral on the fixed locus using the
virtual localization formula \cite{gploc}.
\begin{defn}\label{bps}
Let $M_i^T$ be connected component of $T$-fixed locus $M_d(k)^T$.
Let $N_i^{\rm vir}$ be the virtual normal bundle to $M_i^T$ obtained
from the moving part of the virtual tangent space. We define the
\emph{equivariant genus zero equivariant BPS invariant} by
\[n_d(k)=\sum_i\int_{[M_i^T]^{\rm
 vir}}\frac{1}{e(N_i^{\rm vir})}.\]
Here, $e(-)$ is the equivariant Euler class.
\end{defn}

Note that since $M_d(k)$ has a symmetric obstruction theory and $T$ preserves the Calabi-Yau form, all dual weights in the localization formula cancel each other and the resulting invariant $n_d(k)$ is a number.

In Section \ref{residue}, we show the moduli space $M_d(k)$ is smooth of dimension $kd^2+1$. So, the BPS invariant in Definition \ref{bps} is given by the signed topological Euler characteristic
\[n_d(k)=(-1)^{kd^2+1} e_{\rm top} (M_d(k)).\]

The following is standard. See for example \cite{CG}.
\begin{thm}\label{thm:localization}
Let $M$ be a quasi-projective $\cc$-scheme of finite type. Let $T$ be an algebraic torus acting regularly on $M$. Then $e_{\rm top}(M) = e_{\rm top}(M^T)$.
\end{thm}

\begin{prop}[\cite{kool, kaneyama}]\label{prop:tfixedequaltequi}
A stable sheaf on a toric variety supported on a compact subscheme is $T$-fixed if and only if it is
$T$-equivariant.
\end{prop}
\comments{
In these assertions, we need a projective base scheme. As we will
see in the next section, we can consider sheaf on $Y$ supported on
$\po$ as a sheaf on the Hirzebruch surface $\hk$ on which we can
find a line bundle that gives our stability.}

Hence, one can compute the equivariant BPS invariant by counting equivariant sheaves. In the following section, we review the classification of equivariant sheaves in \cite{kool}.

\section{Equivariant Sheaves}\label{equisheaf}

As a toric variety, $Y$ contains a two-dimensional torus $(\cc^*)^2$
which is isomorphic to $T$ by the isomorphism
\[(\cc^*)^2\ni (t_1,t_2) \mapsto (t_1,t_2,t_1^{-1}t_2^{-1})\in T.\]
The action of this torus is the same as the restriction of
$T$-action on $X$ to $Y$. So, by a slight abuse of notation, we also
denote this embedded torus by $T$ and consider $T$-equivariant
sheaves on $Y$. \comments{It is well known that a stable sheaf on $Y$
supported on a compact subscheme is $T$-fixed if and only if it is
$T$-equivariant. See for example \cite{kool, kaneyama}.
In these assertions, we need a projective base scheme. As we will
see in the next section, we can consider sheaf on $Y$ supported on
$\po$ as a sheaf on the Hirzebruch surface $\hk$ on which we can
find a line bundle that gives our stability.}

In this section, we describe pure equivariant sheaves $\ff$ on $Y$
following \cite{kool}. Let $M$ be the group of characters of $T$ and
$N$ be the group of one parameter subgroups. Then, the fan
associated to $Y$ (which lies in $N\otimes \mathbb{R}$) is
\[\{\sigma_1={\rm Cone}((0,1),(1,0)),\sigma_2={\rm
 Cone}((0,1),(-1,-k))\}\] where ${\rm Cone}(v_1,v_2)$ denote the
convex cone generated by vectors $v_1$ and $v_2$. The $T$-invariant
subvariety associated to the face $(0,1)$ is the zero section of
$\op(k)$ (Figure 1).
\begin{figure}\begin{center}
\begin{pspicture}(4,4)%
\rput(2,2){\pspolygon[fillstyle=solid,
fillcolor=lightgray,linestyle=none](0,0)(-1,-2)(-2,-2)(-2,2)(2,2)(2,0)
\psline(0,0)(2,0) \psline(0,2)(0,-2) \psline(0,0)(-1,-2)
\rput(1,1){$\sigma_1$}\rput(-1,0){$\sigma_2$} }
\end{pspicture}\end{center}
\caption{Toric fan of $Y$}
\end{figure}
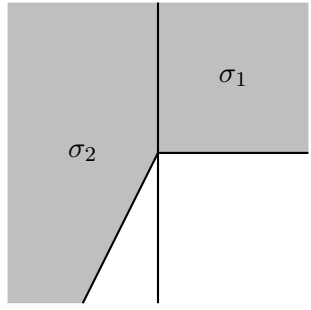

We have two $T$-invariant affine open sets $U_{\sigma_i}={\rm
Spec}(k[S_{\sigma_i}])$, $i=1,2$, where $S_{\sigma_i}$ is the
semigroup defined by $\sigma_i$ \[S_{\sigma_i}=\sigma_i^\vee\cap
M.\] For a notational convenience, we let $M^i$ be a copy of $M$
whose elements are expressed with respect to the semigroup generator
of $S_{\sigma_i}$, that is,
\[M^1=\{m_1(1,0)+m_2(0,1)\} \text{ and }
M^2=\{m_1(-1,0)+m_2(-k,1)\}.\] For $m,m'\in M^i$, we say $m'\geq m$
if every component of $m'-m$ is nonnegative. Note that this means $m'-m$ is an element of the
semigroup $S_{\sigma_i}$.

A quasi-coherent sheaf $\ff$ on $U_{\sigma_i}$ corresponds to a $\Gamma(U_{\sigma_i},\oo_{U_{\sigma_i}})$-module $\Gamma(U_{\sigma_i},\ff)$ and under this equivalence a $T$-equivariant structure on $\ff$ corresponds to a $T$-action on $\Gamma(U_{\sigma_i},\ff)$. In turn, this $T$-action gives a decomposition into weight spaces
\[\Gamma(U_{\sigma_i},\ff)=\bigoplus_{m\in M^i}\Gamma(U_{\sigma_i},\ff)_{m}.\]
Denote the weight space $\Gamma(U_{\sigma_i},\ff)_{m}$ by $F^i(m)$,
$m=(m_1,m_2) \in M^i$. Since $\ff$ is $\oo_Y$-module, each
$\Gamma(U_{\sigma_i},\ff)$ is $M^i$-graded $\cc[S_{\sigma_i}]$-module. We
can reformulate the $\cc[S_{\sigma_i}]$-module structure by the following
data: $k$-linear maps $\chi^i_{m,m'}\colon F^i(m)\to F^i(m')$ for
all $m, m'\in M^i$ with $m'\geq m$ such that
\begin{equation}\label{compatibility}
\chi^i_{m,m}=1\text{ and }\chi^i_{m,m''}=\chi^i_{m',m''}\circ
\chi^i_{m,m'}.\end{equation}

Moreover, in our case, where the reduced support of $\ff$ is $\po$,
we have the following \cite[Chapter 2]{kool}.
\begin{prop}\label{class}
Let $\ff$ be a pure equivariant sheaf on $Y$ with support $\po$.
Then,
\begin{enumerate}
  \item There are integers $A_1^1$, $A_1^2$, and $A\leq B$ such that $F^i(m_1,m_2)=0$ unless
  $A_1^i\leq m_1$ and $A\leq m_2\leq B$.
  \item For each $A\leq m_2\leq B$, the maps $\chi^i_{(m_1,m_2),(m_1+1,m_2)}$ are all injective and the direct limit
  $\displaystyle{\varinjlim_{m_1} F^i(m_1,m_2)}$ is a finite-dimensional vector space denoted by $F^i(\infty,m_2)$.
  \item For each $A\leq m_2\leq B$,
  \[F^1(\infty,m_2)\simeq F^2(\infty,m_2)\]
  and under this identification, \[\chi^1_{(\infty,m_2),(\infty,m_2+1)}=\chi^2_{(\infty,m_2),(\infty,m_2+1)},\] where
  $\chi^i_{(\infty,m_2),(\infty,m_2+1)}=\displaystyle{\varinjlim_{m_1}
  \chi^i_{(m_1,m_2),(m_1,m_2+1)}}$.
\end{enumerate}
Moreover, let $\mathcal{C}$ be the category whose objects are
$\{F^i(m), \chi^i_{m,m'}\}$ satisfying above conditions and
morphisms \[\phi\colon\{F^i(m), \chi^i_{m,m'}\}\to \{G^i(m),
\lambda^i_{m,m'}\} \] are collections of linear maps
$\phi^i(m)\colon F^i(m)\to G^i(m)$ satisfying \[ \phi^i(m')\circ
\chi^i_{m,m'}=\lambda^i_{m,m'}\circ\phi^i(m)\text{ and
}\phi^1(\infty,m_2)=\phi^2(\infty,m_2).\] Then, this correspondence
is an equivalence between the category of pure equivariant sheaves
and equivariant morphisms with the category $\mathcal{C}$.
\end{prop}

An object in the categry $\mathcal{C}$ is called
\emph{$\Delta$-family} \cite{per}. Proposition~\ref{class} is a
special case of more general statement about pure equivariant
sheaves on a toric variety \cite{kool}. We state it for $Y$ to avoid
heavy notation.

Assume $A$ in condition $(1)$ is maximally chosen. Let $\oo_Y(\chi)$
be the structure sheaf of $Y$ endowed with the equivariant structure
induced by a character $\chi\in M$. Then $\ff\otimes \oo_Y(\chi)$ is
isomorphic to the sheaf $\ff$ with equivariant structure shifted by
$\chi$. Therefore, up to isomorphism we may take $A=0$. Since we are only interested in the equivariant sheaves up to isomorphism, we will always assume $A=0$.

The following example shows how we can illustrate an equivariant sheaf as a diagram.
\comments{We will use the following convention.
\begin{convention}\label{conv:box}
A box in $M^1$ corresponds to the lattice point of its lower left corner whereas a box in $M^2$ corresponds to the lattice point of its lower right corner. We use this convention because the horizontal coordinate axes are in different directions.
\end{convention}
}
\begin{ex}
Let $C_n$ be the $n$-th order thickening of $\po$ in the direction
of $\oo_\po(k)$. More precisely, $C_n$ is ${\rm Spec}({\rm
Sym}(\oo_\po(-k))/\mathfrak{I})$ where $\mathfrak{I}$ is the ideal
generated by $S^n(\oo_\po(-k))$. Then for the sheaf $\oo_{C_n}$, we
have
\[\Gamma(U_{\sigma_i},\oo_{C_n})_{(m_1,m_2)}=\begin{cases} \cc &\text{ if } 0\leq m_2\leq n-1 \text{ and } m_1\geq
0 \\ 0 &\text{ else}\end{cases}\]

 We can illustrate this by putting a box at the position
$(m_1,m_2)$ if the corresponding weight space is nonzero. By condition (2) and (3) in Proposition \ref{class}, for each open chart, the asymptotic weight vector
spaces are stabilized and identified with each other. So, we place
the asymptotic vector spaces in the middle. For example, the sheaf
$\oo_{C_3}$ can be depicted as in Figure~\ref{oc3}.

\begin{figure}\begin{center}
\begin{pspicture}(6,3)
{\pspolygon[fillstyle=solid,
fillcolor=lightgray,linestyle=none](0.5,0.5)(5.5,.5)(5.5,2)(.5,2)
\psline{->}(0,.5)(2.5,.5)\psline{->}(6,.5)(3.5,.5)
\psline{->}(0.5,0)(.5,3) \psline{->}(5.5,0)(5.5,3)
\rput(2.3,2.7){$U_{\sigma_1}$}\rput(3.8,2.7){$U_{\sigma_2}$} }
\rput(2.5,0.2){$^{(1,0)}$}\rput(0.1,2.7){$^{(0,1)}$}
\rput(3.5,0.2){$^{(-1,0)}$}\rput(6,2.7){$^{(-k,1)}$}
\psline(.5,.5)(2.5,.5)\psline[linestyle=dotted](2.5,.5)(3.5,.5)\psline(3.5,.5)(5.5,.5)
\rput(0,1){
\psline(.5,.5)(2.5,.5)\psline[linestyle=dotted](2.5,.5)(3.5,.5)\psline(3.5,.5)(5.5,.5)}
\rput(0,.5){
\psline(.5,.5)(2.5,.5)\psline[linestyle=dotted](2.5,.5)(3.5,.5)\psline(3.5,.5)(5.5,.5)}
\rput(0,1.5){
\psline(.5,.5)(2.5,.5)\psline[linestyle=dotted](2.5,.5)(3.5,.5)\psline(3.5,.5)(5.5,.5)}
\psline(1,.5)(1,2)\psline(1.5,.5)(1.5,2) \psline(2,.5)(2,2)
\psline(5,.5)(5,2)\psline(4.5,.5)(4.5,2) \psline(4,.5)(4,2)
\end{pspicture}\end{center}
\caption{The sheaf $\oo_{C_3}$}\label{oc3}
\end{figure}
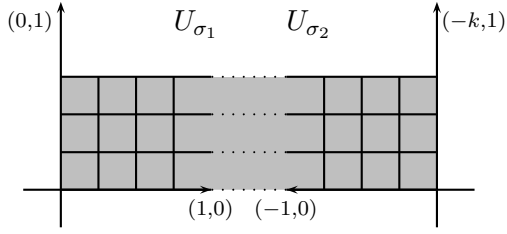

In this particular example, all weight spaces are one dimensional.
We will see other examples in which weight spaces have more than one
dimension. 
\end{ex}

From this description, it is clear that the equivariant version of
Grothendieck's theorem holds.
\begin{thm}\label{decomp}
Let $\ee$ be a equivariant vector bundle of rank $r$ on $\po$. Then
there are integers $a_1,\cdots, a_r$ uniquely determined up to order
such that we have an equivariant isomorphism $\ee\simeq
\oo(a_1)\oplus \cdots \oplus \oo(a_r)$.
\end{thm}
\begin{proof}
This theorem is due to Klyachko \cite{klyachko}. Since the scheme
theoretic support is $\po$, we must have $A=0$ and $B=0$ in the
condition (1) of Proposition~\ref{class}. Let
$(\{E^1(m,0)\},\{E^2(m,0)\})$ be the corresponding family. Then, we
can pick a basis $\{v_j\}$ of the asymptotic weight space
$E^1(\infty,0)\simeq E^2(\infty,0)$ in such a way that for any $m$
and $i=1,2$, a subset of $\{v_j\}$ forms a basis of $E^i(m,0)$.
Therefore, by taking subfamilies generated by each $v_j$,
$(\{E^1(m,0)\},\{E^2(m,0)\})$ decomposes into families with one-dimensional weight spaces. Hence, $\ee$ decomposes equivariantly
into equivariant line bundles.
\end{proof}

Let $\pi\colon Y \to \po$ be the natural projection and $\ff$ be a pure sheaf on $Y$. Then, $\pi_*\ff$ is a locally free sheaf on $\po$, so it has a decomposition \[\pi_*\ff \simeq \displaystyle{\bigoplus_{i=1}^d\oo_\po(a_{i})}.\]
$\pi_*$ induces an equivalence between the category of
$\oo_Y$-modules and the category of $\pi_*\oo_Y$-modules on
$\po$\cite[Ex.II.5.17]{har}. Since $\pi_*\oo_Y\simeq {\rm Sym}(\oo_\po(-k))$,
$\pi_*\oo_Y$-modules structure on $\pi_*\ff$ is given by a map
$$\pi_*\ff\to \pi_*\ff\otimes \oo_\po(k).$$
In what follows, we will state and prove the equivariant version of this correspondence.

Let $\cc^*$ denote the subtorus $\{(t_1,1)\}\subset T$. Then $\cc^*$ naturally acts on $\po$. We need to fix a $\cc^*$-equivariant structure of $\oo_\po(k)$.
Let $U_i$ be the intersection of the open set $U_{\sigma_i}$ with
$\po$ for $i=1,2$. Then, $\{U_i\}$ is an affine open cover of $\po$.
We fix a $\cc^*$-equivariant structure of $\op(k)$ by the weight space
decomposition on each open set
\[\Gamma(U_1,\op(k))=\bigoplus_{m\ge 0}\cc_m, \ \Gamma(U_2,\op(k))=\bigoplus_{m\le k}\cc_m,\]
where $\cc_m$ is a one-dimensional representation of $T$ with
character $\chi(t)=t^m$, $m\in\zz$.

Given a $\cc^*$-equivariant sheaf $\ff$ on $\po$ with
\[\Gamma(U_i,\ff)_m= F^i(m),\]
we will use the natural equivariant structure on $\ff\otimes \op(k)$ given by
\begin{equation}\label{ok}
\begin{array}{l}
\Gamma(U_1,\ff\otimes \op(k))_m= F^1(m) \\
\Gamma(U_2,\ff\otimes \op(k))_m= F^2(m-k).
\end{array}
\end{equation}

Now, let $\ff$ be a $T$-equivariant sheaf on $Y$. We consider $j$th row of the weight space decompositions.
Let $\ff_j$ be the sheaf defined by
\[ \Gamma(U_i,\ff_j)_{(m_1,m_2)}=\begin{cases}  \Gamma(U_i,\ff)_{(m_1,m_2)}&  \text{ if } m_2=j\\ 0 & \text{ else.
}\end{cases}\] Then, $\ff_j$ has scheme theoretic support $\po$ and
hence decomposes into equivariant line bundles by Theorem
\ref{decomp}.

\begin{thm}\label{m}
Let $\cc^*$ denote the subtorus $\{(t_1,1)\}\subset T$. A pure $T$-equivariant sheaf $\ff$ on $Y$ is determined by a collection $\{\ff_j,\phi_j\}$, where each $\ff_j$ is a locally free $\cc^*$-equivariant sheaf on $\po$ for $ 0 \le j\le B$, and $\phi_j \colon \ff_j \to
\ff_{j+1}\otimes \oo_\po(k)$ are $\cc^*$-equivariant maps.

Moreover, let $\ff$ and $\gc$ be $T$-equivariant sheaves on $Y$ corresponding to $\{\ff_j,\phi_j\}$ and
$\{\gc_j,\psi_j\}$ respectively. Then $\ff$ and $\gc$ are isomorphic to each other if and only if there exist isomorphisms $\mu_j : \ff_j \to \gc_j$ such that $\mu_{j+1}\circ\phi_j=\phi_{j+1}\circ\mu_j$.
\end{thm}
\begin{proof}
By Proposition \ref{class}, $\ff$ is given by $\{F^i(m), \chi^i_{m,m'}\}$. For $ 0 \le j\le B$, define $\ff_j$ as above. The horizontal maps $\chi^i_{(m_1,j),(m_1+1,j)}$ endows $\ff_j$ with a $\cc^*$-equivariant structure.
It remains to consider the vertical maps $\chi^i_{(m_1,j),(m_1,j+1)}$.

Recall that we are using different basis of $M$ for $\chi^1$ and
$\chi^2$. The $(m_1,j)$ in the subscript means $m_1(1,0)+j(0,1)$ for
$\chi^1$ and $m_1(-1,0)+j(-k,1)$ for $\chi^2$. Rewrite in the
standard basis of $M$,
\begin{align}
\chi^1_{(m_1,j),(m_1,j+1)}&\colon F^1(m_1,j)\to F^1(m_1,j+1)\notag\\
\chi^2_{(m_1,j),(m_1,j+1)}&\colon F^2(-m_1-kj,j)\to
F^2(-m_1-kj-k,j+1).\notag
\end{align}
Thus, these will define an equivariant morphism
\[\phi_j\colon \ff_j\to\ff_{j+1}\otimes \oo_\po(k)\]
by \eqref{ok}.

Conversely, an equivariant structure on $\ff_j$ will give an weight space decomposition and horizontal maps $\chi^i_{(m_1,j),(m_1+1,j)}$. The equivariant morphisms $\phi_j$ define $\chi^i_{(m_1,j),(m_1,j+1)}$ which commute with horizontal maps. Hence the data $(\ff_j, \phi_j)$ determine $\ff$ by Proposition~\ref{class}.

The second statement is a straightforward consequence of the equivalence between the
category of $\oo_Y$-modules and the category of $\pi_*\oo_Y$-modules on $\po$.
\end{proof}

\begin{rmk}
\begin{enumerate}
  \item We have a decomposition $\ff_j\simeq \bigoplus_{i=1}^{d_j}\oo_\po(a_{ij})$ by Theorem~\ref{decomp}.
Then, it is clear that \[\pi_*\ff\simeq \displaystyle{\bigoplus_{j=0}^B\bigoplus_{i=1}^{d_j}\oo_\po(a_{ij})}.\]\
If $P_\ff(n)=dn+\chi(\ff)$. Then
\[d=\sum_{j=0}^B d_i \text{  and  } \chi(\ff)=\sum_{j=0}^B\sum_{i=1}^{d_j} (a_{ij}+1). \]
  \item We will call the collection $\{\ff_j,\phi_j\}$ associated to a sheaf $\ff$ a \emph{$\pi_*\oo_Y$-module structure} of $\ff$.
  \item By \cite{kaneyama}, a locally free sheaf $\ff$ on $\po$ is $\cc^*$-equivariant if and only if it is $\cc^*$-fixed, that is, if there exists an isomorphism $t^*\ff\simeq\ff$ for all $t\in\cc^*$. Moreover, a map $\phi\colon\ff\to \gc$ between equivariant sheaves are equivariant if and only if $t^*\phi$ is conjugate to $\phi$ with respect to automorphisms of $\ff$ and $\gc$.
  \end{enumerate}
\end{rmk}

To test the stability, we only need to test for equivariant
subsheaves.
\begin{prop}\label{stab}
Suppose $X$ is a projective variety with a torus action. Let $\ff$
be a pure equivariant sheaf on $X$. Then $\ff$ is (Gieseker) stable
if and only if $p_{\mathcal{G}}<p_\ff$ for any proper equivariant
subsheaf $\mathcal{G}$.
\end{prop}
\begin{proof}
See \cite[Proposition 3.19]{kool}
\end{proof}

Therefore, a sheaf $\ff$ associated to $\{\ff_j, \phi_j\}$ is stable
if and only if for any $\pi_*\oo_Y$-submodule $\gc=\{\gc_j,
\psi_j\}$, that is, a collection of equivariant subsheaves
$\mathcal{G}_j\subset\ff_j$ compatible with $\phi_i$, we have
\[\frac{\chi(\mathcal{G})}{r(\mathcal{G})}<\frac{\chi(\ff)}{d},\]
where $r(\mathcal{G})$ is the multiplicity of $\mathcal{G}$ along
$\po$.

\begin{defn}
For a pure equivariant sheaf $\ff$ as in Theorem~\ref{m}, we will
call $(d_0,d_1,\cdots,d_B)$ the \emph{type} of $\ff$.
\end{defn}

\section{Enumeration of Equivariant Sheaves}\label{classification}
Using the classification given in the previous section, we want to
count the (virtual) number of $T$-equivariant sheaves.
\begin{defn}
Let $M_{(d_0,\cdots,d_B)}^T(k)$ denote the subscheme of $M_d(k)$
which consists of stable $T$-equivariant sheaves of type
$(d_0,\cdots,d_B)$ with $d=\sum_{j=0}^B d_j$. We define
\begin{align}
N_d(k)&=e_{\rm top}(M_d(k)),\notag\\
N_{(d_0,\cdots,d_B)}(k)&=e_{\rm
top}(M_{(d_0,\cdots,d_B)}^T(k)).\notag\end{align}
\end{defn}
In Section \ref{residue}, we will see that the BPS invariant $n_d(k)$ is given by the signed topological Euler characteristic. Hence, it suffices to compute $N_d(k)$. By Proposition \ref{prop:tfixedequaltequi}, it is clear from the localization formula that
\begin{equation}\label{localization}
N_d(k)=\sum_{\substack{(d_0,\cdots,d_B)\\d_0+\cdots+d_B=d }}N_{(d_0,\cdots,d_B)}(k).
\end{equation}

\subsection{Type $(1^d)$}

Let $(1^d)$ denote $(1,1,\cdots,1)$ with $1$ repeated $d$ times. Let
$\ff$ be a $T$-equivariant sheaf of type $(1^d)$ whose
$\pi_*\oo_Y$-module structure is $\{\ff_j,\phi_j\}$. Assume
$\ff_j\simeq \op(a_j)$ for $0\le j \le d-1$. Then, since
$\chi(\ff)=1$, we have
\[\sum_{j=0}^{d-1} (a_j+1)=1.\]
Let $x$ and $y$ be homogeneous coordinates of $\po$. By Theorem~\ref{m}, the map $\phi_j$ is given
by a monomial in $x$ and $y$ of degree $a_{j+1}-a_j+k$.

\begin{prop}\label{1d}
$\ff$ of type $(1^d)$ is stable if and only if $\phi_j$'s are all
nonzero and
\[ \sum_{j=0}^{h-1} (a_j+1) \ge 1\]
for any $1\le h\le d$.
\end{prop}
\begin{proof}
Assuming $\ff$ is stable, it is indecomposable and therefore all $\phi_j$'s are nonzero. 
To check the stability, it is enough to check for the subsheaf $\gc$
with
\[
\gc_j=\begin{cases}\ff_j & \text{ if } j\ge h\\
0 &\text{ else}\end{cases}\] for $0\le h\le d-1$. Hence, the
stability condition is
\[ \sum_{j=0}^{h-1} (a_j+1) \ge 1\]
where the left-hand side is the Euler characteristic of $\ff/\gc$.
\end{proof}

\begin{cor}\label{count1}
$N_{(1^d)}$ is equal to
\[\sum_{\lambda_{d-1}\ge\cdots \ge\lambda_0\ge 0}
\prod_{j=0}^{d-2}(\lambda_{j+1}-\lambda_{j}+1)
\]where the sum runs over all $\lambda_{d-1}\ge\cdots \ge\lambda_0\ge
0$ such that
\[\sum_{j=0}^{d-1} \lambda_j=\frac{d(d-1)}{2}k-(d-1)
\] and for any $1\le
h\le d$,
\[\sum_{j=0}^{h-1} \lambda_j\ge \frac{h(h-1)}{2}k-(h-1).\]
\end{cor}
\begin{proof}
Since $\phi_j$ is nonzero, we have $a_j\le a_{j+1}+k$. We let
$$\lambda_j=a_j+jk$$ so that $\lambda_{d-1}\ge \lambda_{d-2}\ge\cdots
\ge\lambda_0\ge 0$. Then, $\phi_j$ is a monomial of degree
\[a_{j+1}-a_j+k =\lambda_{j+1}-\lambda_{j}.\]
Each coefficient of the monomial $\phi_j$ can
be set to be 1 by scaling isomorphisms. So, we have
$\lambda_{j+1}-\lambda_{j}+1$ choices for $\phi_j$. The condition
for $\lambda_j$'s can easily be seen to be equivalent to the
condition in $a_j$'s in the previous proposition.
\end{proof}

\subsection{Types $(n,1^d)$ and $(1^d,n)$}

We will use the following lemma frequently.

\begin{convention}
For a monomial $\alpha$ in $x$ and $y$, we set ${\rm
gcd}(\alpha,0)=\alpha$. Hence, ${\rm deg}({\rm gcd}(\alpha,0))={\rm
deg}(\alpha).$
\end{convention}

\begin{lem}\label{deg}
Suppose
\[\phi=(\alpha_1, \alpha_2)\colon\op(a_1)\oplus \op(a_2)\to\op(b)\otimes \op(k) \]
\[\psi=(\beta_1, \beta_2)^t\colon\op(c) \to(\op(d_1)\oplus\op(d_2))\otimes \op(k)\]
are nonzero maps between sheaves on $\po$ where $\alpha_1$,
$\alpha_2$, $\beta_1$, and $\beta_2$ are monomials of appropriate
degrees in the homogeneous coordinates $x$ and $y$. 
Let $K$ be the
kernel of $\phi$ and $Q$ be such that $Q\otimes \op(k)$ be the
torsion-free part of the cokernel of $\psi$. Then
\begin{align}{\rm deg} K &= a_1+a_2-b-k+{\rm deg}({\rm
gcd}(\alpha_1,\alpha_2)),\label{degk}\\
{\rm deg} Q &= d_1 +d_2 -c+k-{\rm deg}({\rm
gcd}(\beta_1,\beta_2)).\label{degq}
\end{align}

\end{lem}
\begin{proof}
Let $r={\rm deg}({\rm gcd}(\alpha_1,\alpha_2))$. If either of
$\alpha_1$ or $\alpha_2$ is zero, by symmetry, we may assume
$\alpha_1$ is zero. Then, $\alpha_2$ is nonzero monomial of degree
$b-a_2+k$. So, $K\simeq \op(a_1)$ and we have \eqref{degk}. Now,
suppose $\alpha_1$ and $\alpha_2$ are both nonzero. Since $\alpha_i$
is a monomial of degree $b-a_i+k$, there are monomials $p$ and $q$
of degree $b-a_1+k-r$ and $b-a_2+k-r$ respectively, such that
$$q \alpha_1=p\alpha_2.$$ Then the image of the inclusion
$$\xymatrix @1{\op(a_1+a_2-b-k+r) \ar[r]^-{ \left(\begin{subarray}{c} q \\-p \\ \end{subarray} \right)} &
\oo_\po(a_1)\oplus\oo_\po(a_2)}$$ is $K$. Therefore, we have
\eqref{degk}.

The proof of \eqref{degq} is similar.
\end{proof}
\begin{defn}
Given a map $\psi$ as in the above lemma, we will call $Q$ the
\emph{torsion-free cokernel} of $\psi$.
\end{defn}

We start with types $(n,1)$ and $(1,n)$. Let $\ff$ be a
$T$-equivariant sheaf of type $(n,1)$ which corresponds to the
collection $(\{\ff_0, \ff_1\},\phi)$. Assume $\ff_0\simeq
\oplus_{i=1}^n \op(a_i)$ and $\ff_1\simeq \op(b)$ and
$\phi=(\alpha_1,\cdots, \alpha_n)$, where
\[\alpha_i\colon\op(a_i)\to \op(b)\otimes\op(k).\] Then,
$\chi(\ff)=1$ is equivalent to
\begin{equation}\label{chin1}
\sum_{i=1}^{n} (a_i+1)+ (b+1)=1.\end{equation} As before, by Theorem~\ref{m}, $\alpha_i$ is given by a
monomial in $x$ and $y$ of degree $b-a_i+k$ if $b-a_i+k\ge 0$.

\begin{prop}\label{n1}
$\ff$ of type $(n,1)$ is stable if and only if
\begin{itemize}
  \item $\alpha_i$'s are nonzero,
  \item $a_i\geq 0, \ \forall i $,
  \item for all $1\le i, j\le
n$, ${\rm deg}({\rm gcd}(\alpha_i,\alpha_j))\leq b-a_i-a_j+k-1$.
\end{itemize}
\end{prop}
\begin{proof}
As before, for $\ff$ to be indecomposable, $\alpha_i$'s are nonzero. 
Let $\gc$ be a subsheaf of $\ff$ whose $\pi_*\oo_Y$-module structure
is $(\{\gc_0, \gc_1\},\psi)$. Since $\ff_1$ is of rank 1, it suffices to consider the two cases: $\gc_1=\ff_1$ or $\gc_1=0$.

Suppose $\gc_1=\ff_1$. Let $\gc_0\simeq \oplus_{i=1}^{r}\op(a_i')$
where $r$ is the rank of $\gc_0$. Without loss of generality, we may
assume $a_i$'s and $a_i'$'s are nonincreasing. Then for $1\le i\le
r$,
\[a_i'\le a_i\] because otherwise there does not exist an injective map
from $\gc_0$ to $\ff_0$. So, it is enough to check for the cases
$a_i'=a_i$, i.e., $\gc_0\simeq \oplus_{i=1}^{r}\op(a_i)$ for some
$0\le r\le n$. Then by looking at the quotients, it is easy to see
that the stability implies $a_i\geq 0$ for all $1\le i\le n$. Note
that if $\gc_0=0$, we have $b\le -1$ which is a consequence of
\eqref{chin1} and $a_i\geq 0$.

Now suppose $\gc_1=0$. Then, $\gc_0$ is a subsheaf of $K={\rm
ker}\phi$. Let $K_{ij}$ be the kernel of the restricted map
\[(\alpha_i,\alpha_j)\colon \op(a_i)\oplus\op(a_j) \to \op(b)\otimes \op(k).\]
When $\gc_0=K_{ij}$, by \eqref{degk}, the stability implies
\[a_i+a_j-b-k+{\rm deg}({\rm gcd}(\alpha_i,\alpha_j))\leq -1,\] which is the third condition.
For an arbitrary $\gc_0$, it suffices to show the degree of $ \gc_0$
is negative provided that the degrees of $K_{ij}$ are negative for
all $1\le i,j\le n$. We may assume that $\gc_0$ is a line
bundle. By Proposition~\ref{stab}, we may also assume $\gc_0$ is
equivariant subsheaf of $\ff_0$, that is, the inclusion
\[\gc_0\to
\bigoplus_{i=1}^{n}\op(a_i)\simeq \ff_0
\]
is given by a matrix with monomial entries. We write $(p_1,\cdots,p_n)^t$
for the inclusion map, where $p_i$'s are monomials. Then, we have
\[
\sum_{i=1}^n p_i\alpha_i=0.
\]
Since all terms are monomials and at least two terms are nonzero,
this implies that there exist $j_1$ and $j_2$ such that
$p_{j_1}\alpha_{j_1}$ and $p_{j_2}\alpha_{j_2}$ are nonzero and
proportional. Then ${\rm deg}(p_{j_1}\alpha_{j_1})\ge {\rm deg}({\rm
lcm}(\alpha_{j_1}, \alpha_{j_2}))$, and
\begin{align} {\rm deg}\ \gc_0&= a_{j_1}-{\rm deg}(p_{j_1})\le
a_{j_1}+{\rm deg}(\alpha_{j_1})-{\rm
deg}({\rm lcm}(\alpha_{j_1}, \alpha_{j_2}))\notag\\
&= a_{j_1}-{\rm deg}(\alpha_{j_2})+{\rm
deg}({\rm gcd}(\alpha_{j_1}, \alpha_{j_2}))\notag\\
&= a_{j_1}+a_{j_2}-b-k+{\rm deg}({\rm gcd}(\alpha_{j_1},
\alpha_{j_2}))\notag\\
&= {\rm deg} \ K_{j_1,j_2}\le -1.\notag
\end{align}
Hence it is enough to check for subsheaves
$K_{ij}$.\end{proof}

The type $(1,n)$ is dual to the type $(n,1)$. Now, assume
$\ff_0\simeq \op(c)$ and $\ff_1\simeq \oplus_{i=1}^n\op(d_i)$ and
$\phi=(\beta_1,\cdots, \beta_n)^t$, where
\[\beta_i\colon\op(c)\to \op(d_i)\otimes\op(k)\]
is given by a monomial in $x$ and $y$ of degree $d_i-c+k$. Then,
$\chi(\ff)=1$ is equivalent to \begin{equation}\label{chi1n}
(c+1)+\sum_{i=1}^{n} (d_i+1)=1.\end{equation}

\begin{prop}\label{1n}
$\ff$ of type $(1,n)$ is stable if and only if
\begin{itemize}
  \item $\beta_i$'s are nonzero,
  \item $d_i\leq -1, \ \forall i $,
  \item for all $1\le i, j\le
n$, ${\rm deg}({\rm gcd}(\beta_i,\beta_j))\leq d_i+d_j-c+k$.
\end{itemize}
\end{prop}

\begin{proof}
The proof is dual to the proof of the previous proposition.
Indecomposability implies $\beta_i$'s are nonzero. Let $\gc$ be a
subsheaf of $\ff$ whose $\pi_*\oo_Y$-module structure is $(\{\gc_0,
\gc_1\},\psi)$. If $\gc_0=0$, it is enough to check for $\gc_1\simeq
\op(d_i)$. So, we have $d_i\leq -1$.

Suppose $\gc_0=\op(c)$. Let $Q_{ij}$ be the torsion-free cokernel of
the map
\[\beta_{ij}=(\beta_i,\beta_j)^t\colon \op(c)\to (\op(d_i)\oplus\op(d_j))\otimes\op(k).\]
If $\gc_1$ is the saturation of $\bigoplus_{t\ne i,j}
\op(d_t)\oplus ({\rm im} \beta_{ij}\otimes \op(-k))$, then by \eqref{degq},
\[{\rm deg} Q_{ij} = d_i +d_j -c+k-{\rm deg}({\rm
gcd}(\beta_i,\beta_j))\ge 0,\] which is the third condition. Now,
let $\gc_1$ be an arbitrary subsheaf of $\ff_1$ containing the image
of $\phi$. We may assume $\gc_1$ is an equivariant saturated
subsheaf of rank $n-1$. Let $(q_1,\cdots,q_n)$ be the natural
projection map from $\ff_1$ to the quotient $\ff_1/\gc_1$ where
$q_i$'s are monomials. Then
\[\sum_{i=1}^n \beta_i q_i=0.\]
As in the previous proposition, we can find $j_1$ and $j_2$ such that
$\beta_{j_1}q_{j_1}$ and $\beta_{j_2}q_{j_2}$ are nonzero and
proportional. Thus,
\begin{align}
\text{deg} \ \ff_1/\gc_1 &=d_{j_1}+{\rm deg}(q_{j_1})\ge
d_{j_1}+{\rm deg}(\beta_{j_2})-{\rm
deg}({\rm gcd}(\beta_{j_1}, \beta_{j_2}))\notag\\
&= d_{j_1}+d_{j_1}-c+k-{\rm deg}({\rm gcd}(\beta_{j_1},
\beta_{j_2}))\notag\\
&= {\rm deg} \ Q_{j_1,j_2}\ge 0.\notag
\end{align}
So, it is enough to check for $Q_{ij}$.
\end{proof}

Propositions~\ref{n1} and \ref{1n} have straightforward
generalizations to types $(n,1^d)$ and $(1^d,n)$.

\begin{prop}\label{n1d}
For a sheaf $\ff$ of type $(n,1^d)$, let $\ff_0\simeq \oplus_{i=1}^n
\op(a_i)$ and $\ff_j\simeq \op(b_j)$ for $1\le j\le d$. By
$\chi(\ff)=1$, we have
\begin{equation}
\sum_{i=1}^{n} (a_i+1)+ \sum_{j=1}^d(b_j+1)=1.\end{equation} Then,
$\ff$ is stable if and only if
\begin{itemize}
  \item  all maps $\phi_j$,
$0\le j\le d$ have nonzero monomial entries,
  \item $a_i\geq 0, \ \forall i $, $\sum_{j=s}^d (b_j+1) \le 0$ for $1\le s\le
  d$,
  \item ${\rm deg}({\rm gcd}(\alpha_i,\alpha_j))\leq
  b_1-a_i-a_j+k-1$,
\end{itemize}
where $\phi_0=(\alpha_1,\cdots, \alpha_n)$.
\end{prop}
\begin{prop}\label{1dn}
For a sheaf $\ff$ of type $(1^d,n)$, let $\ff_j\simeq \op(c_j)$ for
$0\le j\le d-1$ and $\ff_d\simeq \oplus_{i=1}^n\op(d_i)$. By
$\chi(\ff)=1$, we have\begin{equation}
\sum_{j=0}^{d-1}(c_j+1)+\sum_{i=1}^{n} (d_i+1)=1.\end{equation}
Then, $\ff$ is stable if and only if
\begin{itemize}
  \item all maps
$\phi_j$, $0\le j\le d$ have nonzero monomial entries,
  \item $d_i\leq -1, \ \forall i $, $\sum_{j=0}^s (c_j+1) \ge 1$ for $0\le s\le d-1$,
  \item ${\rm deg}({\rm gcd}(\beta_i,\beta_j))\leq
  d_i+d_j-c_{d-1}+k$,
\end{itemize}
where $\phi_{d-1}=(\beta_1,\cdots, \beta_n)^t$.
\end{prop}

\begin{cor}\label{isolated}
All stable equivariant sheaves of type $(1^d)$, $(n,1^d)$ or
$(1^d,n)$ are isolated points in $M_d(k)^T$.
\end{cor}
\begin{proof}
By scaling automorphisms in each case, we can set the coefficients
of monomials to be 1. So, equivariant sheaves of these types are
isolated.
\end{proof}
\begin{cor}\label{1nn1}
For any $k\ge -1$,
\[N_{(1,n)}(k)=N_{(n,1)}(k+n-1)\]
\end{cor}
\begin{proof}
For a given $c$ and $d_j$'s as in Proposition~\ref{1n}, we let
$a_j=-1-d_j$ and $b=-n-c$ and $\alpha_j=\beta_j$. Note that
$\text{deg}(\beta_j)=d_j-c+k=b-a_j+(k+n-1)$ as required. Moreover,
$b-a_i-a_j+((k+n-1)-1)=d_i+d_j-c+k$, and equation \eqref{chin1}
for the Euler characteristic is equivalent to \eqref{chi1n}. So,
$a_j$, $b$, and $\alpha_j$ so defined will determine a stable sheaf in
$M_{(n,1)}^T(k+n-1)$. Hence, this gives a bijection between
$M_{(1,n)}^T(k)$ and $M_{(n,1)}^T(k+n-1)$.
\end{proof}

\section{The Calculation of BPS invariants}\label{bpscounting}
In this section, we compute the local BPS invariants when $d=1,2,3$,
and $4$. Conjecture~\ref{GV} combined with the
Gromov-Witten theory \cite{localgw} predicts that
\begin{align}
&n_1(k)=(-1)^{k+1},\label{eq:n1} \\
&n_2(k)=
\begin{cases}
 -\frac{k(k+2)}{4} & \text{if $k$ is even,} \\
 -\frac{(k+1)^2}{4} & \text{if $k$ is odd,}
\end{cases}\\
&n_3(k)=(-1)^{k+1}\frac{k(k+1)^2(k+2)}{6}.\\
&n_4(k)=-\frac{k(k+1)^2(k+2)(2k^2+4k+1)}{12}. \label{n4}
\end{align}

The genus zero Gromov-Witten invariants are obtained by expanding the formula in \cite[Corollary 7.2]{localgw} with $g=0$. Then, we obtain \eqref{eq:n1}-\eqref{n4} inductively by applying \eqref{gv}. In Section~\ref{residue}, we will show that 
\begin{equation}\label{sign}
n_d(k)=(-1)^{kd^2+1}e_{\rm top}(M_d(k)).
\end{equation} Hence, the signs are correct and it is enough to compute $N_d(k)=e_{\rm top}(M_d(k))$. By the localization formula \eqref{localization}, we compute $N_d(k)$ by counting $T$-equivariant sheaves.

\subsection{$d=1$}
By Corollary~\ref{count1}, it is easy to see that $N_1(k)=1$. We can
see this more directly. Let $\ff$ be a stable sheaf with Hilbert
polynomial $n+1$ whose support is $\po$. Then $\ff$ has a section,
or a nonzero morphism $\oo_\po \to \ff$. Since $\oo_\po$ is stable
with the Hilbert polynomial $n+1$, this morphism is an isomorphism.
Hence $$ M_1^T(k)=\{\oo_\po\}.$$ Hence, we have
\[ N_1(k)=1.\]

\subsection{$d=2$}
Only sheaves of type $(1,1)$ appear. By Corollary~\ref{count1},
\[N_2(k)=\sum_{\lambda_1\ge \lambda_0} (\lambda_1- \lambda_0+1)\]
where the sum is over all partitions $\lambda_1+ \lambda_0=k-1$.
Therefore,
\[
N_2(k)=\sum_{\lambda_0=0}^{\lfloor\frac{k-1}{2}\rfloor}
(k-2\lambda_0)=
\begin{cases}
\frac{k(k+2)}{4} & \text{if $k$ is even.} \\
\frac{(k+1)^2}{4} & \text{if $k$ is odd.}
\end{cases}
\]

\subsection{$d=3$}
In this case, sheaves of type $(1,1,1)$, $(2,1)$, and $(1,2)$ appear.
By Corollary~\ref{1nn1},
\begin{equation}\label{n12}
N_{(1,2)}(k)=N_{(2,1)}(k+1).\end{equation} We start with the type
$(2,1)$.

To count the $T$-equivariant sheaves of type (2,1), we let
\[S_{(2,1)}(k)=\left\{(a_1,a_2,b)\in \zz^3\colon \begin{array}{l}
a_1+a_2+b=-2,\\
0\le a_1\le b+k, ~ 0\le a_2\le b+k \end{array}\right\}.\]

For $(a_1,a_2,b)\in S_{(2,1)}(k) $, we count pairs
$(\alpha_1,\alpha_2)$ of nonzero monomials with no common factor of
degree greater than $b-a_1-a_2+k-1=2b+k+1$.
\begin{defn} For $r< {\rm min}(n,m)$, define
\[P_{(n,m,r)}=\left\{(v,w)\colon \begin{array}{l}v,w \text{ monomials
in } x\text{ and }y\\
{\rm deg} v=n,{\rm deg} w=m, {\rm deg}({\rm gcd}(v,w))\leq
r\end{array}\right\}.\]
\end{defn}

\begin{lem}\label{f}
$|P_{(n,m,r)}|=\begin{cases}(r+1)(r+2)& \text{if }0 \le r< {\rm
min}(n,m)\\0& \text{if }r<0\end{cases}$
\end{lem}
\begin{proof}
For $(v,w)\in P_{(n,m,r)}$, let $g$ be ${\rm gcd}(v,w)$ and $d$ be
its degree. Then $(v,w)$ is either $(x^{n-d}g, y^{m-d}g)$ or
$(y^{n-d}g, x^{m-d}g)$. Since there are $d+1$ choices for $g$,
$|P_{(n,m,r)}|$ is $2\sum_{d=0}^r(d+1)=(r+1)(r+2)$. \end{proof}

 Note that if $a_1=a_2=a$, switching two factors
of $\ff_0=\op(a)\oplus\op(a)$ gives an isomorphism between two
sheaves determined by $(\alpha_1,\alpha_2)$ and
$(\alpha_2,\alpha_1)$. So, we must count half of such pairs
$(\alpha_1,\alpha_2)$ if the degree of $\alpha_1$ and $\alpha_2$ are
the same.

We let
\[f(n,m,r)=\begin{cases}|P_{(n,m,r)}|& \text{if }n\neq m \\
\frac{1}{2}|P_{(n,m,r)}|& \text{if }n= m
\end{cases}\]

Then, the total number of $T$-fixed sheaves of type $(2,1)$ is
\begin{equation}\label{eq:n21}
N_{(2,1)}(k)=\sum_{(a_1,a_2,b)\in
S_{(2,1)}(k)}f(b-a_1+k,b-a_2+k,2b+k+1).
\end{equation}

\begin{lem}\label{n21}
If $k\ge 1$,
\begin{equation*}
N_{(2,1)}(k)=\sum_{b=\lceil-\frac{k+1}{2}\rceil}^{-1} \lfloor -\frac{b+1}{2}\rfloor (k+2b+2)(k+2b+3)+ \frac{1}{2}\sum_{a=0}^{\lfloor\frac{k-3}{4}\rfloor} (k-4a-2)(k-4a-1)
\end{equation*}
\end{lem}
\begin{proof}
Each sum corresponds to the case $a_1>a_2$ and $a_1=a_2$
respectively. Note that in \eqref{eq:n21}, $f$ has $\frac{1}{2}$
factor if and only if $a_1=a_2$.

First, we count the case $a_1>a_2$. From equation \eqref{eq:n21},
since $r=2b+k+1\ge 0$, $-\frac{k+1}{2}\leq b\leq -1$. For each $b$,
we can check there are $\lfloor -\frac{b+1}{2}\rfloor$ pairs of
$(a_1,a_2)$ with $a_1>a_2$ satisfying all the required conditions.
By Lemma~\ref{f}, this verifies the first sum.

If $a_1=a_2=a$, then $b=-2-2a \geq -\frac{k+1}{2}$. So, $0\leq a\leq
\frac{k-3}{4}$. Thus by \eqref{eq:n21} and Lemma~\ref{f}, we obtain the
second sum.
\end{proof}

Now, to count sheaves of type $(1,1,1)$, let
\[S_{(1,1,1)}(k)=\left\{(\lambda_0,\lambda_1,\lambda_2)\in \zz^3\colon
\sum_{i=1}^3\lambda_i=3k-2, 0\le
\lambda_0\le\lambda_1\le\lambda_2\le 2k-1\right\}\]

Then, by Corollary~\ref{count1},
\begin{equation}\label{eq:n111}
N_{(1,1,1)}(k)=\sum_{(\lambda_0,\lambda_1,\lambda_2)\in
S_{(1,1,1)}(k)}(\lambda_2-\lambda_1+1)(\lambda_1-\lambda_0+1).
\end{equation}

\begin{thm}
\begin{equation}\label{eq:nk}
N_3(k)=\frac{k(k+1)^2(k+2)}{6}
\end{equation}
\end{thm}
\begin{proof}
We compute $N_3(k)-N_3(k-1)$ and prove \eqref{eq:nk} by induction.
It remains to count type (1,1,1) sheaves.

The map
\[(\lambda_0,\lambda_1,\lambda_2)\mapsto (\lambda_0+1,\lambda_1+1,\lambda_2+1)\] gives an injection from
$S_{(1,1,1)}(k-1)$ to $S_{(1,1,1)}(k)$. Since the summand in
\eqref{eq:n111} does not change under this map, the corresponding
terms cancel each other in $N_3(k)-N_3(k-1)$.

The remaining terms in $ N_3(k)$ are for $\lambda_0=0$ or
$\lambda_2=2k-1$. We claim that
\[
N_{(1,1,1)}(k)-N_{(1,1,1)}(k-1)
=\sum_{\lambda_1=k-1}^{\lfloor\frac{3k-2}{2}\rfloor}(3k-2\lambda_1-1)(\lambda_1+1)
+\sum_{\lambda_0=1}^{\lfloor\frac{k-1}{2}\rfloor}(\lambda_0+k+1)(k-2\lambda_0).
\]

If $\lambda_0=0$, then we must have $\lambda_1+\lambda_2=3k-2$,
$\lambda_2\le 2k-1$. So, $\lambda_2=3k-2-\lambda_1$ and $k-1\leq
\lambda_1\leq \frac{3k-2}{2}$. Hence we have the first term.

If $\lambda_0\neq 0$ and $\lambda_2=2k-1$, we must have
$\lambda_0+\lambda_1=k-1$, and $\lambda_0>0$. So, $1\leq
\lambda_0\leq \frac{k-1}{2}$, which verifies the second term.

Now, using Lemma~\ref{n21} and \eqref{n12}, we can check case by
case ($k$ mod 4) that
\[N_3(k)-N_3(k-1)=\frac{k(k+1)(2k+1)}{3}.\] Since it is easy to verify
\eqref{eq:nk} for small values of $k$, this proves the theorem.
\end{proof}
\begin{cor}
Conjecture~\ref{GV} holds for $d=1,2,$ and $3$.
\end{cor}

\subsection{$d=4$}\label{degree4}
Types $(1,1,1,1)$, $(3,1)$, $(1,3)$, $(2,1,1)$ and
$(1,1,2)$ are treated in Section~\ref{classification}. The remaining
types are $(1,2,1)$ and $(2,2)$. In these types, positive-dimensional torus fixed loci can occur.

\begin{ex}\label{nonisolated}
We give an example of a positive-dimensional $T$-fixed locus in
degree 4 of type $(1,2,1)$ when $k=2$.

Let $\ff_0=\op$, $\ff_1=\op(-1)\oplus \op(-1)$ and $\ff_2=\op(-1)$.
The $\pi_*\oo_Y$-module structure is
\[\phi_0 =
\left(
  \begin{array}{c} x\\y\\
  \end{array}
\right) \colon \op\to (\op(-1)\oplus \op(-1))\otimes \op(2),\]
\[\phi_1=\left(
  \begin{array}{cc} c_1y^2 & c_2xy \\
  \end{array}
\right) \colon \op(-1)\oplus \op(-1)\to \op(-1)\otimes \op(2),\]
where $c_1$ and $c_2$ are in $\cc$. It is easy to see that $\phi_0$ and $\phi_1$ are $\cc^*$-equivariant map as in Theorem \ref{m}. As only scaling isomorphisms are allowed, we cannot set all coefficients to be 1 using isomorphisms.

Let $\ff(c_1,c_2)$ be such
a sheaf. By Proposition \ref{121}, one can see that $\ff(c_1,c_2)$ is stable if $(c_1,c_2)\ne (0,0)$. Since we have $\ff(c_1,c_2)\simeq\ff(\lambda c_1,\lambda c_2)$ for $\lambda \in
\cc^*$, this $T$-fixed locus is isomorphic to $\po$.
\end{ex}

Let the \mo structure of a sheaf $\ff$ of type (1,2,1) be
\begin{align}
\phi_0=(\alpha_1, \alpha_2)^t&\colon\op(a)\to(\op(b_1)\oplus \op(b_2))\otimes \op(k) \label{module121}\\
\phi_1=(\beta_1, \beta_2)&\colon\op(b_1)\oplus \op(b_2) \to\op(c)\otimes \op(k),\notag
\end{align}
where $\alpha_i$ and $\beta_i$ are monomials with coefficient 1.

\begin{prop}\label{121}
Without loss of generality, assume $b_1\ge b_2$. The data \eqref{module121} define a stable equivariant sheaf $\ff$ with $\chi(\ff)=1$ if and only if:
\begin{enumerate}
  \item $\alpha_1\beta_2$ and $\alpha_2\beta_1$ are proportional.
  \item $a+b_1+b_2+c=-3$.
  \item No more than one of $\alpha_1,\alpha_2,\beta_1,$ or $\beta_2$ is zero.
  \item $c\le -1$, $a\ge 0$, $b_1+c \le -2$.
  \item ${\rm deg}({\rm gcd}(\alpha_1,\alpha_2))\le b_1 +b_2-
  a+k$,\\
${\rm deg}({\rm gcd}(\beta_1,\beta_2))\leq c+k-b_1-b_2-1$.
  \item If $\alpha_1\beta_1+\alpha_2\beta_2=0$, then ${\rm deg}({\rm gcd}(\beta_1,\beta_2))\leq
c+k-b_1-b_2-a-2.$\label{cond4}
\end{enumerate}
\end{prop}
\begin{proof}
For \eqref{module121} to define an equivariant sheaf, the composition \[ (\phi_1\otimes id_{\op(k)})\circ\phi_0\colon \op(a)\to\op(c)\otimes \op(2k) \] is also given by a monomial by Theorem \ref{m}. In other words, $\alpha_1\beta_2$ and $\alpha_2\beta_1$ are proportional. Hence, item 1. 
We can easily see that the condition $\chi(\ff)=1$ is equivalent to item 2. 

The items 3-6 are for the stability of $\ff$.

If at least two of $\alpha_1,\alpha_2,\beta_1,$ and $\beta_2$ are zero, $\ff$ is decomposable.
Suppose $\gc=(\gc_0,\gc_1,\gc_2)$ is a $\pi_*\oo_Y$-submodule, where
$\gc_0\subset\op(a)$, $\gc_1\subset\op(b_1)\oplus \op(b_2)$ and
$\gc_2\subset\op(c)$. Let
$${\rm rank} (\gc)=({\rm rank} (\gc_0),{\rm rank} (\gc_1),{\rm rank}
(\gc_2)).$$ For each possible choice of the rank of $\gc$, we
examine the stability condition.
\begin{enumerate}
  \item ${\rm rank} (\gc)=(0,0,1)$ : $c\le -1$.
  \item ${\rm rank} (\gc)=(0,2,1)$ : $b_1+b_2+c\le -3$ or $a\ge 0$ by item 2.
  \item ${\rm rank} (\gc)=(0,1,1)$ : Since
the degree of $\gc_1$ is no more than $b_1$ as $b_1\ge b_2$, we have
$b_1+c \le -2$.
\item ${\rm rank} (\gc)=(1,1,1)$ : We can reduce to the case when $\ff/\gc$ is
the torsion-free cokernel of $\phi_0$. So, by Lemma~\ref{deg},
stability condition is
\[b_1 +b_2 - a+k-{\rm deg}({\rm gcd}(\alpha_1,\alpha_2))\ge 0.\]
\item ${\rm rank} (\gc)=(0,1,0)$ : The kernel of $\phi_1$ has degree
$b_1 +b_2-c-k+{\rm deg}({\rm gcd}(\beta_1,\beta_2)).$ So, \[{\rm
deg}({\rm gcd}(\beta_1,\beta_2))\leq c+k-b_1-b_2-1.\]
\item ${\rm rank} (\gc)=(1,1,0)$ : A subsheaf of this type exists
only if the image of $\phi_0$ is in the kernel of $\phi_1$, i.e., if
$\alpha_1\beta_1+\alpha_2\beta_2=0$. In such a case, we take
$\gc_0=\op(a)$ and $\gc_1= {\rm ker} \phi_1$. So,
\[a+\left(b_1 +b_2-c-k+{\rm
deg}({\rm gcd}(\beta_1,\beta_2))\right)\leq -2,\] which is item 6.
\end{enumerate}
\end{proof}

Let the \mo structure of a sheaf $\ff$ of type $(2,2)$ be
\begin{equation}\label{module22}
\phi=\left(\begin{array}{cc} \phi_{11} & \phi_{12} \\ \phi_{21} & \phi_{22}\\ \end{array}\right)
\colon\op(a_1)\oplus\op(a_2)\to(\op(b_1)\oplus \op(b_2))\otimes
\op(k),\end{equation}
where all $\phi_{ij}$ are monomials with coefficient 1.


\begin{prop}\label{22}
Without loss of generality, assume  $a_1\ge a_2$ and $b_1\ge b_2$. The data \eqref{module22} define a stable equivariant sheaf $\ff$ with $\chi(\ff)=1$ if and only if:
\begin{enumerate}
  \item $\phi_{11}\phi_{22}$ and $\phi_{12}\phi_{21}$ are proportional.
  \item $a_1+a_2+b_1+b_2 = -3.$
  \item $\phi_{21}$ is nonzero. No more than one of $\phi_{ij}$ is zero.
  \item $a_1\ge a_2\ge 0$ and $ b_2\le b_1 \le-1$.
  \item ${\rm deg}({\rm
gcd}(\phi_{11},\phi_{21}))\le a_2+b_1+b_2-a_1+k+1$, \\
${\rm deg}({\rm gcd}(\phi_{21},\phi_{22}))\le b_2-b_1-a_1-a_2+k-2$.
\item If $\phi_{11}\phi_{22}=\phi_{12}\phi_{21}$, then\\
$
{\rm deg}({\rm
gcd}(\phi_{11},\phi_{21}))\le b_1+b_2-a_1+k$, and \\
$ {\rm deg}({\rm gcd}(\phi_{11},\phi_{12}))\le b_1+k-a_1-a_2-1.$
\end{enumerate}
\end{prop}
\begin{proof}
For \eqref{module121} to define an equivariant sheaf, the matrices $\phi$ and $t.\phi$ for $t\in T$ must be conjugate with each other under automorphisms of $\op(a_1)\oplus\op(a_2)$ and $\op(b_1)\oplus \op(b_2)$, which can be seen equivalent to item 1.
We can easily see that the condition $\chi(\ff)=1$ is equivalent to item 2.

The items 3-6 are for the stability of $\ff$.

If at least two of $\phi_{ij}$ are zero, $\ff$ is decomposable.

Let
\begin{align}r_1={\rm deg}({\rm gcd}(\phi_{11},\phi_{12})), \hspace{1em} & r_2={\rm
deg}({\rm gcd}(\phi_{21},\phi_{22})),\notag\\ s_1={\rm deg}({\rm
gcd}(\phi_{11},\phi_{21})),\hspace{1em} & s_2={\rm deg}({\rm
gcd}(\phi_{12},\phi_{22})).\notag
\end{align}
Then by item 1,
\begin{equation}\label{gcd}r_2=r_1+b_2-b_1\text{ and }
s_2=s_1+a_1-a_2,\end{equation} provided that all $\phi_{ij}$ are
nonzero.

 Suppose $\gc=(\gc_0,\gc_1)$ is a
$\pi_*\oo_Y$-submodule. For each possible choice of the rank of
$\gc$, we examine the stability conditions.
\begin{enumerate}
  \item ${\rm rank} (\gc)=(0,1)$ : $ b_1\le -1 $.
  \item ${\rm rank} (\gc)=(0,2)$ : $b_1+b_2\le -2$.
  \item ${\rm rank} (\gc)=(1,2)$ : $a_2\ge 0$.
  \item ${\rm rank} (\gc)=(1,1)$ : Let $\gc_0=\op(m)$ and
  $\gc_1=\op(n)$. If $a_2<  m\le  a_1$, $\gc_0$ is a subsheaf of
  $\op(a_1)$. Hence, we can replace $\gc_0$ by $\op(a_1)$ and take
  $\gc_1$ to be the saturation of the image of $\op(a_1)$ under
  $\phi$. The quotient is $(\op(a_2),$ the torsion-free cokernel of $\phi |_{\op(a_1)})$.
  So, for $\ff$ to be stable, we must have
  \[a_2+b_1+b_2-a_1+k-s_1\ge -1,\] by Lemma~\ref{deg}.
  Note that if $\phi_{21}$ is zero, $s_1=b_1-a_1+k$. Then the
  quotient has degree $a_2+b_2$. This contradicts the stability, since we have $a_2+b_2\le -2$ by items 2 and 4 of the proposition which are proved above. 

  Now suppose $m\le a_2$. If $n\le b_2$, since $a_2+b_2\le -2$, there
  is nothing to check. If $b_2< n\le b_1$, we can replace $\gc_1$
  by $\op(b_1)$ and take $\gc_0$ to be the inverse image of $\op(b_1)$, that is, the kernel of the map \[(\phi_{21},\phi_{22})\colon
  \op(a_1)\oplus \op(a_2) \to \op(b_2)\otimes \op(k).\]
  Then the condition is
  \[b_1+ a_1+a_2-b_2-k+r_2\le -2.\]
  \item ${\rm rank} (\gc)=(1,0)$ or $(2,1)$ : A subsheaf of these types exists
  only if the image of $\phi$ has rank 1, in other words, if
  $\phi_{11}\phi_{22}=\phi_{12}\phi_{21}$.
Then, by Lemma \ref{deg} and \eqref{gcd}, the torsion-free cokernel of $\phi$ has degree
\[b_1+b_2-a_1+k-s_1=b_1+b_2-a_2+k-s_2, \]
and the kernel of $\phi$ has degree
\[a_1+a_2-b_1-k+r_1=a_1+a_2-b_2-k+r_2. \]
Hence the conditions are
\[s_1\le b_1+b_2-a_1+k \text{ and } r_1\le b_1+k-a_1-a_2-1.\]
\end{enumerate}
\end{proof}

\begin{rmk}
As we will see in the next example, all positive-dimensional loci of
type $(1,2,1)$ can be expressed as a GIT quotient of $(\po)^4$ by
the action of $SL_2(\cc)$. While the linearization may be different,
the quotient is always isomorphic to $\po$. Similar argument for
type $(2,2)$ holds. So, we can see that all $T$-fixed loci in degree
4 are either isolated points or $\po$.
\end{rmk}

\begin{ex}
In Example~\ref{nonisolated}, $\ff_0$, $\ff_1$ and $\ff_2$ are
unchanged along the one-dimensional torus fixed locus. Condition
\eqref{cond4} in Propositions~\ref{121} and \ref{22} suggests that this
is not true in general.

Assume $k=3$ and let $\ff_0=\op(1)$, $\ff_1=\op(-1)\oplus \op(-1)$
and $\ff_2=\op(-2)$. The $\pi_*\oo_Y$-module structure is
\[\phi_0 =
\left(
  \begin{array}{c} x\\y\\
  \end{array}
\right) \colon \op(1)\to (\op(-1)\oplus \op(-1))\otimes \op(3),\]
\[\phi_1=\left(
  \begin{array}{cc} c_1xy & c_2x^2 \\
  \end{array}
\right) \colon \op(-1)\oplus \op(-1)\to \op(-2)\otimes \op(3),\]
where $c_1$ and $c_2$ are in $\cc$. By Proposition~\ref{121}, we can
check the corresponding sheaf $\ff(c_1,c_2)$ is stable unless
$c_1=-c_2$.

As the $T$-fixed locus $M_d(k)$ is compact, the limit of the above
family at $c_1=-c_2$ exists in $M_d(k)^T$. To see what the limit is,
we need to examine $\Delta$-family described in Proposition
\ref{class}.

Assume that the fixed point in the open set $U_{\sigma_1}$ is given by
$x=0$ and the fixed point in $U_{\sigma_2}$ by $y=0$. Then the above
\mo structure has weight space decomposition as Figure~\ref{ex121}.
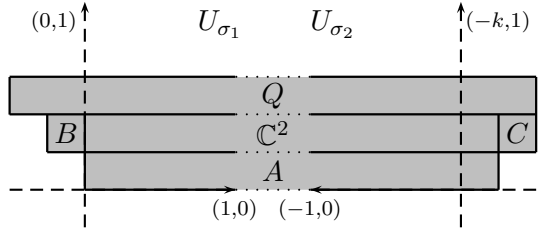
\begin{figure}\begin{center}
\begin{pspicture}(7,3)
\rput(.5,0){ %
\pspolygon[fillstyle=solid,
fillcolor=lightgray,linestyle=none](0.5,0.5)(6,.5)(6,1)(6.5,1)(6.5,2)(-.5,2)(-.5,1.5)(0,1.5)(0,1)(.5,1)(.5,.5)
\psline[linestyle=dashed]{->}(-.5,.5)(2.5,.5)\psline[linestyle=dashed]{->}(6.5,.5)(3.5,.5)
\psline[linestyle=dashed]{->}(0.5,0)(.5,3)
\psline[linestyle=dashed]{->}(5.5,0)(5.5,3)
\rput(2.3,2.7){$U_{\sigma_1}$}\rput(3.8,2.7){$U_{\sigma_2}$}
\rput(2.5,0.2){$^{(1,0)}$}\rput(0.1,2.7){$^{(0,1)}$}
\rput(3.5,0.2){$^{(-1,0)}$}\rput(6,2.7){$^{(-k,1)}$}
\psline[linewidth=.3mm](.5,.5)(2.5,.5)\psline[linestyle=dotted](2.5,.5)(3.5,.5)\psline[linewidth=.3mm](3.5,.5)(6,.5)
\rput(0,.5){
\psline[linewidth=.3mm](0,.5)(2.5,.5)\psline[linestyle=dotted](2.5,.5)(3.5,.5)\psline[linewidth=.3mm](3.5,.5)(6.5,.5)}
\rput(0,1){
\psline[linewidth=.3mm](-.5,.5)(2.5,.5)\psline[linestyle=dotted](2.5,.5)(3.5,.5)\psline[linewidth=.3mm]
(3.5,.5)(6.5,.5)} \rput(0,1.5){
\psline[linewidth=.3mm](-.5,.5)(2.5,.5)\psline[linestyle=dotted](2.5,.5)(3.5,.5)\psline[linewidth=.3mm](3.5,.5)(6.5,.5)}
\psline[linewidth=.3mm](6,.5)(6,1.5)\psline[linewidth=.3mm](6.5,1)(6.5,2)
\psline[linewidth=.3mm](-.5,2)(-.5,1.5)
\psline[linewidth=.3mm](0,1.5)(0,1)\psline[linewidth=.3mm](.5,1.5)(.5,.5)
\rput(.25,1.25){$B$}\rput(6.25,1.25){$C$}\rput(3,1.75){$Q$}\rput(3,.75){$A$}\rput(3,1.25){$\cc^2$}
 }
\end{pspicture}\end{center}
\caption{Sheaf of type $(1,2,1)$}\label{ex121}
\end{figure}

In Figure~\ref{ex121}, $A,B,C,$ and $Q$ are one dimensional. By
Proposition~\ref{class}, $T$-fixed sheaves with such weight space
decomposition are determined by inclusions of $A,$ $B,$ and $C$ into
$\cc^2$ and a surjection $\cc^2\to Q$. The $SL_2(\cc)$ action on
$\cc^2$ via change of basis encodes isomorphism between sheaves. See
\cite[Chapter 3]{kool} for a detailed discussion.

We identify $\cc^2\to Q$ with its kernel $K$ so that $A,B,C,$ and $K$
are in $Gr(1,\cc^2)\simeq \pp^1$. We want to relate Gieseker
stability to GIT stability condition for the action of $SL_2(\cc)$
on $(\pp^1)^4$. It can be checked that the associated sheaf is
Gieseker stable unless \begin{equation}\label{gieseker}A=B\text{ or
}A=C\text{ or }A=K\text{ or }B=C=K.\end{equation} Meanwhile, a point
$(p_1,p_2,p_3,p_4)\in (\pp^1)^4$ is GIT stable with respect to a
line bundle $\oo(k_1, k_2, k_3,k_4)$ if and only if for any point $p
\in \po$
\begin{equation}\label{git}\sum_{p=p_i} k_i < \frac{1}{2}\sum_{i=1}^{4}k_i.\end{equation}
See \cite[Theorem 11.2]{dolgachev}, \cite[Section 4.4]{mfk}. If we
take $k_1=2$, $k_2=k_3=k_4=1$, these two conditions agree with each
other when we let $(p_1,p_2,p_3,p_4)=(A,B,C,K)$. This is an example
of matching GIT stability and Gieseker stability discussed in
\cite[Chapter 3]{kool}.

Therefore, the $T$-fixed locus is
\[(\po)^4\sslash SL_2(\cc)\simeq \po.\]
All positive-dimensional fixed loci can be analyzed similarly.

\comments{The quotient map $(\po)^4\to \po$ is given by the
\emph{cross-ratio}
\[r(p_1,p_2,p_3,p_4)=\frac{(p_2-p_1)(p_4-p_3)}{(p_3-p_1)(p_4-p_2)},\]
where $p_1,p_2,p_3,p_4$} 

The condition $c_1=-c_2$ is equivalent to
$A=K$. It is easy to check that at the limit in $(\po)^4\sslash
SL_2(\cc)$, we have $B=C$ and $A, B,$ and $K$ are distinct. By reading
equivariant vector bundles in each rows, we can see that the limit has
\mo structure
\[\phi_0 =
\left(
  \begin{array}{c} xy\\1\\
  \end{array}
\right) \colon \op(1)\to (\op\oplus \op(-2))\otimes \op(3),\]
\[\phi_1=\left(
  \begin{array}{cc} x & x^2y \\
  \end{array}
\right) \colon \op\oplus \op(-2)\to \op(-2)\otimes \op(3).\] Note
that since $xy$ is a multiple of $1$, or $x^2y$ is a multiple of
$x$, we can set all the coefficients of monomials to be 1 up to
isomorphism.
\end{ex}
\begin{rmk}
Based on the classification of $T$-equivariant stable sheaves
studied above and in Section~\ref{classification}, we can compute
$N_4(k)$. The author has verified that the result is consistent with
\eqref{n4} when $k\le 100$ using a computer program. However, we
do not have a proof for general $k$.
\end{rmk}

\section{Equivariant Residue}\label{residue}
In this section, we compute the virtual tangent space of $M_d(k)$ and verify the signs in the BPS invariants.
The virtual tangent space at $\ff\in M_d(k)$ is
\[ \Ext^1_X(\ff, \ff)- \Ext^2_X(\ff,\ff).\]
Since $T$ preserves the canonical Calabi-Yau form, the canonical
bundle on $X$ is trivial with trivial weight. By equivariant Serre
duality,
$$ \Ext^1_X(\ff, \ff)\simeq \Ext^2_X(\ff,\ff)^*$$
as $T$-representations. So, the dual weights of the moving parts
will be canceled and we just count signs.

Let $\hk$ be the Hirzebruch surface whose toric fan has ray
generators $u_1=(-1,k)$, $u_2=(0,1)$, $u_3=(1,0)$, and $u_4=(0,-1)$.
Denote the corresponding divisors by $D_1$, $D_2$, $D_3$, and $D_4$.
The total space $Y$ of $\oo_\po(k)$ can be described as a toric
variety by the fan $\{ {\rm Cone}(u_3, u_4),{\rm Cone}(u_4,u_1)\}$.
Hence, $Y$ is a subvariety of $\hk$ and the zero section of $Y$ is
the divisor $D_4$. Let $i\colon Y \to \hk$ be the inclusion.

By the equivalence
\[D_1\sim D_3 \text{ and } D_2\sim D_4 -k D_3, \]
any divisor on $\hk$ can be expressed as $aD_3 +bD_4$ for integers $a$ and $b$. We have
\[aD_3 +bD_4 \text{ is ample if and only if } a,b>0.  \]

We fix an ample line bundle $D=2D_3+D_4$. Then, we have a well-defined moduli space $$M_\hk(d)= \{ \ff \text{ sheaf on } \hk \colon c_1(\ff)=dD_4, \chi(\ff)=1, D\text{-(semi)stable}\}.$$
Let $\ff$ be a sheaf on $Y$ supported on a curve of class $d\linecl$. Then $i_*\ff$ is supported on a curve of class $dD_4$. Then, since $D_4\cdot D=k+2$, we have
\[\chi(i_*\ff\otimes \oo(nD))=d(k+2)n +\chi(\ff) = P_\ff((k+2)n),\]
where $P_\ff(n)$ is the Hilbert polynomial defined in \eqref{eq:hilblocalp1}.
Thus, since $k+2> 0$, $i_*\ff$ is $D$-semistable if and only if
$\ff$ is semistable. Hence, $i_*$ induce an injective morphism
from $M_d(k)$ to $M_\hk(d)$.

\begin{prop}\label{prop:smoothness}
$M_\hk(d)$ is a smooth projective scheme of dimension $kd^2+1$.
\end{prop}
\begin{proof}
By Serre duality, $\Ext^2(\ff,\ff)=\Hom(\ff,\ff\otimes K)^*.$ Since $$c_1(\ff)\cdot K= dD_4\cdot K=-d(k+2)<0,$$ we have $\chi(\ff\otimes K)<\chi(\ff)$ by the Riemann-Roch theorem, while $r(\ff\otimes K)=r(\ff)$. Hence, by stability of $\ff$, we have $\Hom(\ff,\ff\otimes K)=0$.
Therefore, there is no obstruction and $M_\hk(d)$ is a smooth projective scheme.

We compute the dimension of $\Ext^1(\ff,\ff)$ using the Riemann-Roch theorem.
$$\chi(\ff,\ff)=1-{\rm dim}\Ext^1(\ff,\ff)=\int_\hk {\rm ch}^\vee(\ff){\rm ch}(\ff){\rm td}(\hk)$$
Since the rank of $\ff$ is zero and $c_1(\ff)= d D_4$, the degree 2
term of right-hand side is $ - d^2 D_4^2=-kd^2.$ Therefore,
$${\rm dim}\Ext^1(\ff,\ff)=1-\chi(\ff,\ff)=kd^2+1.$$
Thus, ${\rm dim} M_\hk(d)=kd^2+1.$
\end{proof}

\begin{cor}
\[n_d(k)=(-1)^{kd^2+1}e_{\rm top}(M_d(k))\]
\end{cor}
\begin{proof}
$M_d(k)$ is open subscheme of $M_\hk(d)$, hence smooth of dimension
$kd^2+1$. Then, this is a consequence of general properties of
Donaldson-Thomas-type invariants with symmetric obstruction theory
\cite{behrend}.
\end{proof}


\end{document}